\documentclass[12pt,letterpaper]{article}
\usepackage{graphicx}
\usepackage{amsmath}
\usepackage{amsfonts}
\usepackage{amsthm}
\usepackage{amssymb}
\usepackage{color}
\def\R{\mathbb{R}}
\usepackage{verbatim} 
\usepackage{hyperref}

\newtheorem{theorem}{Theorem}[section]
\newtheorem{lemma}[theorem]{Lemma}
\newtheorem{corollary}[theorem]{Corollary}

\newcommand{\be}{\begin{equation}}
\newcommand{\ee}{\end{equation}}
\newcommand{\bea}{\begin{eqnarray}}
\newcommand{\eea}{\end{eqnarray}}
\newcommand{\beas}{\begin{eqnarray*}}
\newcommand{\eeas}{\end{eqnarray*}}

\begin{document}

\title{Halving Lines and Their Underlying Graphs}

\author{Tanya Khovanova\\MIT \and Dai Yang\\MIT}
\maketitle

\begin{abstract}
In this paper we study underlying graphs corresponding to a set of halving lines. We establish many properties of such graphs. In addition, we tighten the upper bound for the number of halving lines.
\end{abstract}

\section{Introduction}

Halving lines have been an interesting object of study for a long time. Given $n$ points in general position on a plane the minimum number of halving lines is $n/2$. The maximum number of halving lines is unknown. The current upper bound of $O(n^{4/3})$ is proven by Dey \cite{Dey98}. The current lower bound of $O(ne^{\sqrt{\log n}})$ is found by Toth \cite{Toth}.

We approached the subject of halving lines by studying the properties of the underlying graph. We start with introducing some definitions and examples in Section~\ref{sec:definitions}. In Section~\ref{sec:basic} we discuss some basic properties of underlying graphs.  In Section~\ref{sec:constructions} we introduce important geometric constructions: segmentarizing, cross, and Y-shape.

We continue discussing properties of underlying graphs in Section~\ref{sec:properties}, where we prove any graph can be an induced subgraph of an underlying graph. We also study particular subgraphs of an underlying graph and show that an underlying graph with $n$ vertices can contain $n-1$ path, and $n-3$ cycles at the most and provide a construction to show that the bound is exact. We give an example of an underlying graph containing a clique of size at least $\sqrt{n/2}$. 

We continue by studying chains in Section~\ref{sec:chains}. The chain methods allow us to prove more properties of underlying graphs. In particular, we show that the largest clique cannot exceed the size of $\sqrt{2n}+1$.

We finish in Section~\ref{sec:upperbound} by improving the upper bound on the number of halving lines. We show that the previous upper bound can be tightened by dividing it by $\sqrt[3]{4}$.

\section{Definitions}\label{sec:definitions}

Let $n$ points be in general position in $\R^2$, where $n$ is even. A \textit{halving line} is a line through 2 of the points that splits the remaining $n-2$ points into two sets of equal size.  A halving line in literature can also refer to a line through none of the $n$ points that split the points into two sets of equal size. Under the latter definition, we say that two halving lines are \textit{equivalent} if they produce the same two sets of points, and we are interested in counting the number of such equivalence classes.

We note that there is a bijection between the equivalence classes of halving lines through no points, and the halving lines through 2 points: for each equivalence class, we can consider the supremum of how far counterclockwise we can orient the lines in the equivalence class. This supremum must be a halving line through 2 points.

From now on, we will use the term halving line to refer to a line through 2 points.

The \textit{halving difference} of a line is the difference of the number of points on each side of the line. Sometimes we will produce construction that do not disturb the halving difference of some lines. That means, we add the same number of points on both side of the line and the difference is preserved.

\subsection{The underlying graph}

From our set of $n$ points, we can determine an \textit{underlying graph} of $n$ vertices, where each pair of vertices is connected by an edge iff there is a halving line through the corresponding 2 points.

\subsection{Geometric Graphs}\label{sec:geographs}

In dealing with halving lines, we consider notions from both Euclidean geometry and graph theory. Given a set of points, the underlying graph of halving lines is uniquely defined. Given a graph, it is not clear which set of points produced this underlying graph. A random embedding of the underlying graph on the plane will not respect the structure of halving lines. 

Sometimes when we talk about underlying graphs of configurations of points we want to keep in mind a particular configuration of points. We want each vertex to remember from which point on the plane it came from. We want to have a definition of a graph that remembers the geometric structure of the set of points.

We define a \textit{geometric graph}, to be a pair of sets $(V,E)$, where $V$ is a set of points on the coordinate plane, and $E$ consists of pairs of elements from $V$. In essence, a geometric graph is a graph with each of its vertices assigned to a distinct point on the plane. We will abbreviate geometric graphs as \textit{geographs}.

\subsection{Examples}

\subsubsection{Four points}

Suppose we have four non-collinear points. If their convex hull is a quadrilateral, then there are two halving lines. If their convex hull is a triangle, then there are three halving lines. Both cases are shown on Figure~\ref{fig:4points}.

\begin{figure}[htbp]
\begin{center}
\includegraphics[scale=0.5]{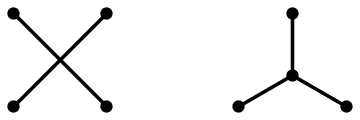}
\end{center}
  \caption{Underlying graphs for four points.}\label{fig:4points}
\end{figure}

\subsubsection{Polygon}

If all points belong to the convex hull of the point configuration, then each point lies on exactly one halving line. The number of halving lines is $n/2$, and the underlying graph is a matching graph --- a union of $n/2$ disjoint edges. The left side of Figure~\ref{fig:4points} shows an example of this configuration.

\subsubsection{Star}

If our point configuration is a regular $(n-1)$-gon and a point at its center, then the underlying graph is a star with a center of degree $n-1$. The configuration has $n-1$ halving lines. The right side of Figure~\ref{fig:4points} shows an example of this configuration.

\section{Basic Properties of the Underlying Graph}\label{sec:basic}

First, we would like to translate basic properties of halving lines into corresponding statements about the underlying graph. Each point has a halving line passing through it. Consequently,

\begin{lemma}
The underlying graph does not have isolated vertices.
\end{lemma}

Each point on the convex hull lies on exactly one halving line. As a consequence the following lemma is true.

\begin{lemma}
The underlying graph has at least three leaves.
\end{lemma}

In particular, it cannot have a Hamiltonian or an Eulerian cycle.

\begin{theorem}
Each vertex of the underlying graph has an odd degree.
\end{theorem}

\begin{proof}
Given a vertex $V$, consider one of the halving lines passing through it. The vertex $V$ splits the line into two rays, one of which contains another vertex $P_1$. Color the ray that contains $P_1$ red, and color the other ray blue. Now consider the instances where this line coincides with the other vertices as we rotate it about $V$.

Let the halving lines through $V$ be $VP_1$, $VP_2$, $\ldots$, $VP_k$, in the order that they coincide with the rotating line. We claim that if $P_i$ meets the red half of the line, then $P_{i+1}$ meets the blue half of the line, and vice versa. Without loss of generality suppose that $P_i$ meets the red half of the line. After $P_i$ leaves the line, the side of the line that $P_i$ is on will have more points than the opposite side of the line. Now consider the next point that meets the rotating line. If it meets the red half of the line, this difference increases, whereas if it meets the blue half of this line, the difference decreases. When $P_{i+1}$ meets the line, the difference must become 0, so $P_{i+1}$ must meet the blue half of the rotating line.

After the line has rotated 180 degrees, it will meet $P_1$ again, but on its blue half. Therefore $k$ must be odd, as desired.
\end{proof}

As part of the proof, we derived the following result:
\begin{corollary}\label{thm:half-plane}
Given two halving lines $VP$ and $VQ$ sharing a single vertex, there exists another halving line $VR$ such that $R$ lies in the opposite angle of $\angle PVQ$. Equivalently, the vectors $\vec{VP},\vec{VQ},\vec{VR}$ do not all lie on a single half-plane.
\end{corollary}

\subsection{Number of halving lines}

As each vertex has at least one halving line passing through it, the minimum number of halving lines is $n/2$. This number is achieved as shown in the example of a convex polygon. The maximum number of halving lines is more difficult to establish. 

The current upper bound of $O(n^{4/3})$ is proven by Dey \cite{Dey98}. The current lower bound of $O(ne^{\sqrt{\log n}})$ is found by Toth \cite{Toth}.

The small examples of up to 26 points were counted by Abrego at al \cite{Abrego} and are available as sequence A076523 in the Online Encyclopedia of Integer Sequences \cite{OEIS}: 1, 3, 6, 9, 13, 18, 22, 27, 33, 38, 44, 51, 57 --- the maximal number of halving lines for $2n$ points on the plane.

Any number of halving lines between the lower bound and the upper bound is achievable as the following theorem states.

\begin{theorem}
For a fixed $n$, if there exists two configurations with $k_1$ and $k_2$ halving lines respectively, then for all $l$ such that $k_1 \le l \le k_2$, there exists a configuration with $l$ halving lines. 
\end{theorem}

\begin{proof}
Let configuration $C_1$ and $C_2$ have $k_1$ and $k_2$ halving lines, respectively. We will show that we can move the points in $C_1$ one at a time until we reach $C_2$, in such a way that the number of halving lines never changes by more than 1.

To do so, note that, as we move a point $P$ in $C_1$ to its desired location in $C_2$, the number of halving lines changes only when $P$ crosses a line formed by 2 other points. If we move $P$ in such a way that it never lies collinear with 2 or more pairs of points, the number of halving lines can only increase or decrease by 1 each time $P$ crosses a line formed by 2 other points. Hence, for every $l$, we achieve a configuration with $l$ halving lines at some point in the moving process.
\end{proof}

\section{Constructions}\label{sec:constructions}

There are some constructions that we will use many times, so we will describe them separately in this chapter.

\subsection{Segmenterizing}

Suppose we have a set of points. Any affine transformation does not change the set of halving lines. Sometimes it is useful to picture that our points are squeezed into a long narrow rectangle. This way our points are almost on a segment. We call this procedure \textit{segmenterizing}. The Figure~\ref{fig:squeezing} shows three pictures. The first picture has six points, that we would squeeze towards the line $y=0$. The second picture shows the configuration squeezed by a factor of 10, and if we make the factor arbitrary large the points all lie very close to a segment as shown on the last picture.

\begin{figure}[htbp]
\begin{center}
\includegraphics[scale=0.7]{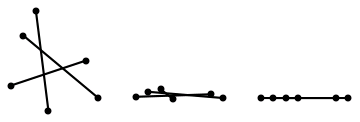}
\end{center}
  \caption{Segmenterizing.}\label{fig:squeezing}
\end{figure}

Note that we do not need to squeeze along the direction that is perpendicular to the segment. As any affine transformation preserves the halving lines, we can use any direction for squeezing.

This procedure makes all the points very close to a single line segment and all the halving lines very close to this line. If we add a point not too close to this line than it lies on the same side of all the halving lines. Moreover, it lines on the same side of all the lines connecting any two original points.

The goal of this construction is to have a lot of space outside the proximity of the segment where we can add new points and have control how halving lines behave.

\subsection{Cross}

The following construction we call a \textit{cross}. 

Here how we do it. We squeeze initial sets into long narrow segments. Then we intersect these segments at middle lines, so that half of the points of each segment lie on one side of all halving lines that pass through the points of the other segment (See Figure~\ref{fig:cross}).

Given two sets of points with $n_1$ and $n_2$ points respectively whose underlying graphs are $G_1$ and $G_2$, the cross is the construction of $n_1+n_2$ points on the plane whose underlying graph has two isolated components $G_1$ and $G_2$. 

\begin{figure}[htbp]
\begin{center}
\includegraphics[scale=0.5]{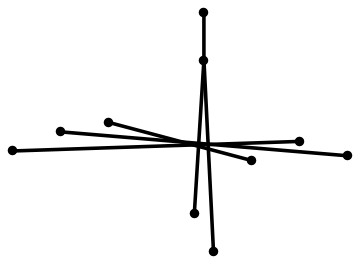}
\end{center}
  \caption{The Cross construction.}\label{fig:cross}
\end{figure}

We have several immediate consequences as a result of this construction.

\begin{lemma}
Any odd degree between 1 and $n-1$ can appear in and underlying graph of $n$ vertices. Any number of connected components between 1 and $n/2$ inclusive can appear in an underlying graph of $n$ vertices.
\end{lemma}

\begin{proof}
The cross of a star graph with $2k$ degrees and a conves polygon with $n-2k$ degrees has $n/2 -k+ 1$ connected components. It has $n-1$ leaves and one vertex of degree $2k-1$.
\end{proof}

\subsection{The Y-shape construction}

Suppose we have three configurations $G_1$, $G_2$, and $G_3$ with $n$ points each and $k_1$, $k_2$, and $k_3$ halving lines correspondingly. The Y-shape construction allows to build a new configuration with $3n$ points which has each of the three initial configuration as a subgraph and has a total of $k_1 + k_2 + k_3 + 3n/2$ halving lines.

The construction works as follows. We segmentarize each set of points $G_i$. Then we draw three rays emanating from the origin, forming an angle 120 degrees between each other, and place each segmentarized set of points along one of the rays. (See Figure~\ref{fig:Yshape}.) This makes a Y-shape of $3n$ points, with $n$ points on each branch.

\begin{figure}[htbp]
\begin{center}
\includegraphics[scale=0.5]{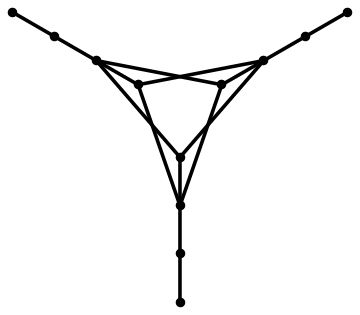}
\end{center}
  \caption{The Y-shape construction.}\label{fig:Yshape}
\end{figure}

All halving lines that served a branch remain to be halving lines. In addition, we can find halving lines that go through 2 points on different branches of the Y-shape. There are a total of $\frac{3}{2}n$ such lines, so we have produced a configuration with $3n$ points and $k_1 + k_2 + k_3+\frac{3}{2}n$ halving lines.

Note that we can use the Y-shape construction on three copies of the same configuration with $n$ points and $k$ halving lines. It produces a new configuration with $3k+\frac{3}{2}n$ halving lines.

One can use the Y-shape construction repeatedly to produce an elementary lower bound of $n\log n$ for the maximum number of halving lines.

\section{Properties of the Underlying Graph}\label{sec:properties}

Now we will discuss the properties of the underlying graph in more detail.

\subsection{Degree sequence}

The \textit{degree sequence} of a graph is the non-increasing sequence of its vertex degrees. The Erdos-Gallai theorem \cite{EG} describes which sequences could be degree sequences of graphs:

\begin{theorem}
A non-decreasing sequence of $n$ numbers $d_i$ is the degree sequence of a simple graph iff the total sum of degrees is even and 
$$\sum_{i=1}^{k}d_i \leq k(k-1) + \sum_{i=k+1}^n \min(d_i,k) \quad \text{for } k \in \{1,\dots,n\}.$$
\end{theorem}

Using the way the degrees interact with each other we can prove the following lemma.

\begin{lemma}\label{thm:atmost3}
At most one vertex can have degree $n-1$, at most three vertices can have degree $n-3$.
\end{lemma}

\begin{proof}
The vertex of degree $n-1$ must connect to all other vertices, hence it must connect to three leaves that we have. Hence only one such vertex can exist. The vertex of degree $n-3$ must connect to at least one out of our three leaves.
\end{proof}

We can make a stronger statement about the vertex of degree $n-1$.

\begin{lemma}
If the underlying graph has a vertex of degree $n-1$, then it is a star graph.
\end{lemma}

\begin{proof}
Consider the configuration of points that resulted in the vertex having degree $n-1$. Make that point the origin. We will show that any other point in this configuration must have degree 1. Indeed, each of these remaining $n-1$ points must lie on one of $n-1$ lines through the origin, occupying every other ray that these $n-1$ lines determine (there are $2n-2$ rays in total). Then it is easy to show using an analysis of slopes that any line not passing through the origin will have more points on the origin's side than on the other side.
\end{proof}

\begin{lemma}
Any degree sequence that contains all threes and ones and at least 3 ones is achievable as an underlying graph of some configuration.
\end{lemma}

\begin{proof}
The degree sequence with 3 ones and everything else threes corresponds to the configuration in the path construction in Lemma~\ref{thm:path}. This configuration crossed with a matching graph can produce any odd number of ones with the rest being threes. 

To achieve an even number of ones, we can use the following modified version of the path construction:

Replace the two vertices lying on the $y$-axis by two vertices that form a horizontal segment which makes the bottom side of the convex hull. Under this configuration, the four vertices on the convex hull have degree 1, and the remaining vertices have degree 3. 
\end{proof}

\subsection{Paths}

Here we consider the size of non-self-intersecting paths in underlying graphs. A path cannot have more than two leaves, so an easy upper bound for the length of the largest path is $n-1$ vertices. It turns out that this bound is exact:

\begin{lemma}\label{thm:path}
For every $n$, there exists an underlying graph of size $n$ having a path through $n-1$ vertices.
\end{lemma}

\begin{figure}[htbp]
\begin{center}
\includegraphics[scale=0.6]{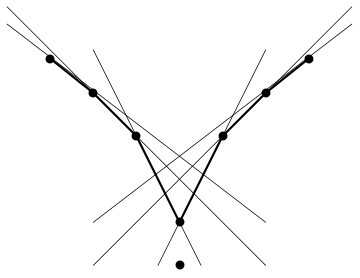}
\end{center}
  \caption{Path.}\label{fig:path}
\end{figure}

\begin{proof}
Figure~\ref{fig:path} shows the path construction for a configuration with 8 points. To avoid clutter only relevant halving lines are shown by thin lines and thick lines show the path in the underlying graph. We generalize this construction to any $n$.

Consider $(n-2)/2$ points that lie on a concave function. Then we segmentarize these points into a segment lying on $x$-axis. Now we place one such segment onto a line $y=x$, to the right of the origin, and another segment on the line $y=-x$ to the left of the origin. We keep the segments oriented in such a way that a line that passes through any two neighboring points of a segment has the remaining $(n-2)/2 -2$ points of the segment below it.
Now add two more points: $(0,-1)$ and $(0,-2)$. Now every line that passes through two neighboring points of a segment becomes a halving line. In addition the point $(0,-1)$ forms halving lines with the rightmost point of the left segment and the leftmost point of the right segment.

The path goes through every point except $(0,-2)$, forming a V-shape.
\end{proof}

\subsection{Cycles}

Here we consider the size of cycles in underlying graphs. Vertices on the convex hull cannot be part of a cycle, so an easy upper bound for the length of the largest cycle is $n-3$. It turns out that this bound is asymptotically exact. We start with a lemma:

\begin{lemma}\label{thm:order}
Suppose a configuration of points with two neighboring points on the convex hull, denoted by $A$ and $B$, is given. We can segmenterize in such a way and choose a direction on the segment so that $A$ becomes the fist point of the segment and $B$ the $k$-th point, for $1 < k \leq n$, where $n$ is the total number of points.
\end{lemma}

\begin{proof}
Draw a line that passes through $A$ and does not pass through $B$. After segmentarization the segment will be on this line. Choose a line through $B$ that has $k-1$ points on the same side as $A$. Segmentarize along this line.
\end{proof}

\begin{theorem}
When $n$ is a multiple of $6$, the maximum length of a cycle is exactly $n-3$.
\end{theorem}

\begin{proof}
We can write $n=3b$, where $b$ is even. Using Lemma~\ref{thm:path}, we can create a configuration of $b$ points with a path of length $b-1$. Note that the endpoints of the path (of the V-shape) are neighboring points on a convex hull. This allows us to use Lemma~\ref{thm:order} to seqmentarize this configuration so that the endpoints of the path occupy the positions 1 and $\frac{b}{2}$.

Now we use three copies of this segment in the Y-shape construction. We orient segments in such a way that the point 1 is oriented closer to the center of the construction. The edges of the $(b-1)$-path inside each branch all remain edges, and we also have edges between the first points of the branches and $\frac{b}{2}$ points connecting all these paths together. This creates a cycle of length $n-3$ as desired. On Figure~\ref{fig:cycle} we demonstrate this cycle for 18 vertices. Note that each branch has 6 vertices and the outmost vertex of each branch does not belong to the cycle.
\end{proof}

\begin{figure}[htbp]
\begin{center}
\includegraphics[scale=0.5]{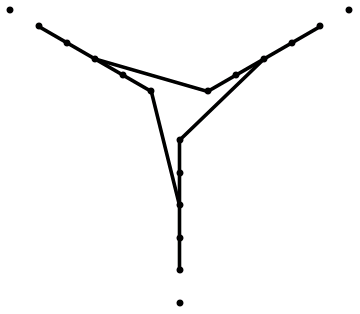}
\end{center}
  \caption{Cycle.}\label{fig:cycle}
\end{figure}

\subsection{Induced Subgraphs}

We just showed that an underlying graph can contain a large path/cycle as a subgraph. If we restrict the graph to the vertices of the path/cycle that we constructed above, we can see that the graph have extra edges in addition to the path/cycle. To differentiate any subgraph from a subgraph that retain all the edges, the notion of induced subgraph is used.

A subgraph $H$ of graph $G$ is said to be an \textit{induced} subgraph if any pair of vertices in $H$ is connected by an edge iff it is connected by an edge in $G$.

\begin{theorem}\label{thm:induced}
Any graph with $2k$ vertices and $e$ edges can be an induced subgraph of an underlying graph with at most $2k+ 2ek-4e+2\binom{2k}{2}$ vertices.
\end{theorem}

\begin{figure}[htbp]
\begin{center}
\includegraphics[scale=0.7]{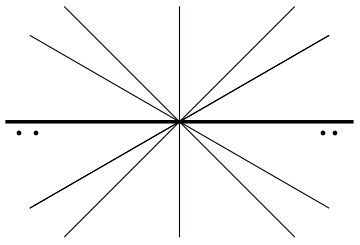}
\end{center}
  \caption{Zooming out and adding points.}\label{fig:induced}
\end{figure}

\begin{proof}
Notice that if the number of vertices is even, then every line has an even halving difference. 
Consider a geograph representing our underlying graph. We will process the configuration line by line. Take a line. Suppose we want to make it a halving line. For this we need to add an even number of points on one side of the line, without disturbing the halving difference of other lines. If it is a halving line and we want to make it a non-halving line we can add 2 points one one side. Let us draw all possible lines connecting the points and zoom out. From a big distance the point configuration will look like a bunch of lines intersecting at one point, see Figure~\ref{fig:induced}. On Figure~\ref{fig:induced} the thick line is the line we are processing. Suppose the line needs an addition of 4 points below it. We add half of the points (two in our example) below the line far away on each side.

Each line that should be an edge in the new underlying graph requires an addition of at most $2k-2$ vertices. All of the future edges require at most $2ek-2e$ extra points. Other lines require at most 2 points each for a total of $2(\binom{2k}{2}-e)$.
\end{proof}

\subsection{Cliques}

Underlying graphs with 4 vertices have cliques of size 2. We can have a clique of size 3 in a graph with six vertices (See Figure~\ref{fig:graph6}).

\begin{figure}[htbp]
\begin{center}
\includegraphics[scale=0.4]{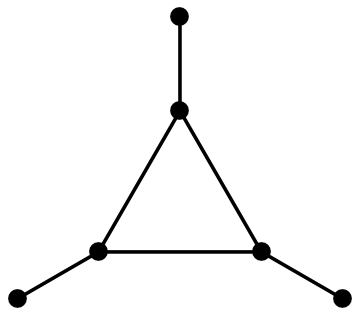}
\end{center}
  \caption{A click of size 3.}\label{fig:graph6}
\end{figure}

By the Theorem~\ref{thm:induced} we can have a clique of size $n$ as an induced subgraph in an underlying graph of size $O(n^3)$. In this subsection we would like to improve the bound by using a construction similar to the construction of Theorem~\ref{thm:induced}, where we process several lines at a time. Clustering lines together allows us to reduce the total number of extra points that we need.

\begin{theorem}
The largest possible clique in an underlying graph of $n$ vertices is at least $O(\sqrt{n})$.
\end{theorem}

\begin{proof}
Let $k$ be even. To produce a clique of size $k$, take a regular $k$-gon, and distort it a little bit using a projective transformation that makes one end of the $k$-gon slightly wider than the other end. This perturbs all the diagonals (and sides) of the $k$-gon that were once parallel, making them intersect somewhere far away from the polygon, but still remain nearly parallel. You can imagine the $k$-gon as drawn on the floor in a painting that respects the perspective properly. This way the lines that are parallel in the $k$-gon, intersect on a point on the horizon line in the painting. We assume that the $k$-gon is in a general position, that is, no two vertices are connected by a line parallel to the horizon. Note that there are now $k$ sets of nearly parallel diagonals and sides, each set having either $k/2$ or $k/2-1$ lines.

\begin{figure}[htbp]
\begin{center}
\includegraphics[scale=0.4]{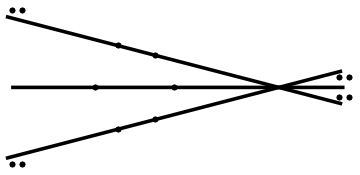}
\end{center}
  \caption{The cluster of nearly parallel lines.}\label{fig:pencil}
\end{figure}

We will now add $O(k^2)$ points to turn this $k$-gon into a $k$-clique. Consider a set of $k/2$ or $k/2-1$ nearly parallel lines. We will process each cluster of lines separately. On Figure~\ref{fig:pencil} we depict one cluster of near parallel sides and diagonals. We rotated the picture so that it fits better in the page, and now the imaginary horizon line is a vertical line through the intersection points on the right. On the half-plane beyond the horizon line add two points between every pair of consecutive lines. This way each line in the cluster becomes a halving line. In addition we want every cluster to be independent. That means we want to add more points so that the halving difference of every line that is not in the cluster does not change. 

We just added to the right of all other lines that are not in the cluster either $k-2$ or $k-4$ points. We need to add the same number of points to the left of all other lines and not to disturb the halving difference we just created in this cluster. The extra points you can see on the picture, they are put in two equal groups on the left above and below the current cluster.

This process requires a total of $2k-4$ or $2k-8$ new points, but turns all of our nearly parallel lines into halving lines without disturbing the other diagonals and sides. Do this a total of $k$ times for each set of nearly parallel lines, and we have constructed an underlying graph with a $k$-clique by adding $2k^2-6k$ points.

Given $n$, we have shown how to construct an underlying graph with a clique of size at least $\sqrt{n/2}$ with no more than $n$ vertices. We can pad this graph to any number of vertices by crossing it with 2-paths.
\end{proof}

We will discuss the upper bound on the size of the clique later.

Any graph can be a subgraph of a clique, so an arbitrary graph with $k$ vertices can always be found as a subgraph, not necessarily induced, of an underlying graph with no more than $O(k^2)$ vertices.

\section{Chains}\label{sec:chains}

We define the following algorithm to group the halving lines into sets that we will call \textit{chains}:
\begin{enumerate}
\item Choose an orientation to define as ``up." The $\frac{n}{2}$ leftmost vertices are called the left half, and the rightmost vertices are called the right half.
\item Start with a vertex on the left half of the graph, and take a vertical line passing through this vertex.
\item Rotate this line clockwise until it either aligns itself with an edge, or becomes vertical again.
\item If it aligns itself with an edge in the underlying graph, define this edge to be part of the chain, and continue rotating the line about the rightmost vertex in the current chain.
\item If the line becomes vertical, we terminate the process. The set of edges in our set is defined as the chain.
\item Repeat step 2 on a different point on the left half of the underlying graph until every edge is part of a chain.
\end{enumerate}

The construction is illustrated in Figure~\ref{fig:chains}. The thickest broken path is the first chain. The next chain is thinner, and the third chain is the thinnest. Note that the chains we get are determined by which direction we choose as ``up."

\begin{figure}[htbp]
\begin{center}
\includegraphics[scale=0.4]{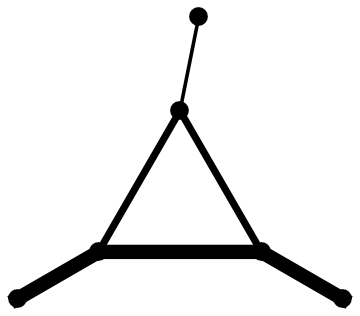}
\end{center}
  \caption{Chains.}\label{fig:chains}
\end{figure}

The following properties of chains follow immediately. Later properties on the list follow from the previous ones:

\begin{itemize}
\item A vertex on the left half of the underlying graph is a left endpoint of a chain.
\item The process is reversible. We could start each chain from the right half and rotate the line counterclockwise instead, and obtain the same chains.
\item A vertex on the right half of the underlying graph is a right endpoint of a chain.
\item Every vertex is the endpoint of exactly one chain.
\item The number of chains is exactly $\frac{n}{2}$.
\item The degrees of the vertices are odd. Indeed, each vertex has one chain ending at it and several passing through it.
\end{itemize}

There are more properties that require some explanation.

\begin{lemma}
Every halving line is part of exactly one chain.
\end{lemma}

\begin{proof}
Suppose there are edges that do not belong to chains. Consider the leftmost vertex and the rightmost vertex in the set of such unattached edges. Suppose the leftmost vertex belongs to the left half-plane. Consider a corresponding unattached edge. Rotate the edge counterclockwise. If we reach a vertical line without crossing other vertices, our edge must be the start of a chain at this vertex. If we do not reach a vertical line, but rather our line passes through another point first, then that point must be to the left of the rotation point and by assumption is part of a chain. By construction our edge is part of the same chain. 

If the leftmost vertex does not belong to the left half-plane, then the rightmost vertex must belong to the right half-plane and we can make a symmetrical argument.

Now that we know that each edge belongs to a chain, we can construct the corresponding chain by starting at this edge. We construct the right part of the chain by rotating the edge around the right vertex clockwise and so on. Similarly, we can recover the left part of the chain.
\end{proof}

\begin{lemma}
The length of each chain is bounded by $\frac{n}{2}$.
\end{lemma}

\begin{proof}
A chain lies below the halving lines corresponding to every edge in the chain. So $n/2 -1$ points that are above each edge in a chain cannot be reached by this chain.
\end{proof}

\begin{lemma}
An underlying graph has at most $2k$ vertices with degree $n-2k+1$.
\end{lemma}

\begin{proof}
The $i$-th vertex from the left in the left half plane can have at most $i-1$ chains passing through it and is a start of exactly one chain. So its degree cannot be more than $2i-1$. Hence, only $k$ rightmost vertices in the left plane and $k$ leftmost vertices in the right plane can have degree $n-2k+1$.
\end{proof}

Before continuing we will introduce wings and windmills.

\subsection{Wings and Windmills}

Let $P$ be a non-ending point of a chain. During step (iii) of making a chain we rotate a line around point $P$. We call the sectors that are covered by this rotating line \textit{wings of the chain at point $P$}. We adjust the definition to include the starting and ending points of the chain. For a starting point, the wing is formed by rotating a vertical line until it reaches the first leg of the chain. Similarly, the wing at the last point is formed by the last leg that rotates into a vertical line.

Wings are formed by two opposite angles, and by the definition of chains, no vertex can have an edge inside its wings. Hence, the following lemma is obvious.

\begin{lemma}[Windmill]
If two chains share a common point $P$, then the two pairs of wings corresponding to each chain do not overlap.
\end{lemma}

We call the two pairs of wings the \textit{windmill}. The windmill is depicted on Figure~\ref{fig:windmill}. The following corollary is immediate:

\begin{figure}[htbp]
\begin{center}
\includegraphics[scale=0.4]{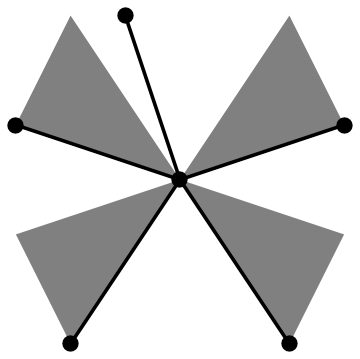}
\end{center}
  \caption{A windmill.}\label{fig:windmill}
\end{figure}

\begin{corollary}
If two chains pass through a common vertex, they cross each other at the vertex.
\end{corollary}

\begin{proof}
The wings in the windmill cross, so the chains cross.
\end{proof}

\begin{lemma}
If a chain and a line pass through the same point $P$ (non-end point), then the chain lies entirely on one side of the line iff the line lies within the wings of the chain at point $P$.
\end{lemma}

\begin{corollary}\label{thm:crossing}
If two chains pass through a common vertex, there is no line such that the chains are on the same side of the line. Similarly, chains cannot be on the opposite sides of the line.
\end{corollary}

We can expand the above results to chains that are made with different orientation.  Remember that for a fixed orientation, any two given points have a unique left-to-right ordering. 

\begin{lemma}
Suppose two chains $S_1$ and $S_2$ under different orientations both contain an edge $AB$. Then
\begin{itemize}
\item If $A$ and $B$ have opposite ordering, then the chains lie on different sides of the line $AB$.
\item If $A$ and $B$ have the same ordering in both chains, then the chains coincide between the rightmost starting point and the leftmost ending point. That is, the union of their edges is a path or a cycle.
\end{itemize}
\end{lemma}

\begin{proof}
If the left-to-right orderings are different, then the convexity of the chains proves the statement. If the left-to-right orderings are the same then the set of edges between the rightmost starting point of the two chains and the leftmost ending points of two chains are identical. This follows from the uniqueness of the chain construction.
\end{proof}

We call a wing at point $P$ a \textit{middle wing} if point $P$ is not a starting or ending point of the chain. The proof of the following lemma is similar to the proofs above.

\begin{lemma}
If point $P$ is a middle point on two chains, possibly under different orientations, then there are only three possibilities:
\begin{itemize}
\item the wings coincide
\item the wings cross
\item the interior of the wings do not overlap, but one side is shared.
\end{itemize}
\end{lemma}

\subsection{The sums of degrees of two vertices}

Now we use our knowledge about chains to refine our knowledge about degrees of the vertices of the underlying graph.

\begin{theorem}
The degrees of two distinct vertices sum to at most $n$, if they are connected by an edge, and at most $n-2$ otherwise.
\end{theorem}

\begin{proof}
Denote the vertices in question as $P$ and $Q$. Rotate the geograph until segment $PQ$ is nearly vertical, so that there are no vertices between the horizontal projections of $P$ and $Q$. If $P$ and $Q$ do not belong to the same chain, then each of $n/2$ chains contributes at most 2 to the sum of degrees of $P$ and $Q$. We have to subtract 2 from this sum since $P$ and $Q$ are both endpoints of some chain(s). Thus the total sum does not exceed $n-2$. 

If $PQ$ is an edge, then it can add two more to the sum of degrees making it at most $n$.
\end{proof}

It immediately follows that the largest clique in the underlying graph cannot be bigger than $n/2$. We can use chains to prove an upper bound on the size of the largest clique that is much closer to the lower bound. But before doing so we would like to introduce some definitions.

\subsection{The straddling span and the largest clique}

Given a line that does not pass through any vertex of a given geograph, we call edges that intersect it \emph{straddling} edges. The maximum number of straddling edges that can be produced by a line is called the \emph{straddling span} of the underlying geograph. Naturally, this notion applies to subgeographs as well. Let us consider some examples.

\begin{lemma}
The straddling span of a $k$-clique is at least $\lfloor k^2/4 \rfloor$.
\end{lemma}

\begin{proof}
Any line that divides vertices of the clique into two equal halves passes through $k^2/4$ edges for an even $k$. If $k$ is odd consider a line that divides vertices almost in half. It passes through $(k^2-1)/4$ edges. Hence the straddling span is at least $\lfloor k^2/4 \rfloor$.
\end{proof}

\begin{lemma}
The straddling span of a $(a,b)$-complete bipartite subgraph is at least $\frac{ab}{2}$.
\end{lemma}

\begin{proof}
Let the partitioned vertices belong to sets $A$ and $B$. Again, we will bound the number of straddling edges. We claim that the number of such edges is minimized when half of the vertices in $A$ are on the left, and the other half are on the right; the same holds for $B$. Indeed, if this is not the case, we can show that trading a vertex in $A$ on one side for a vertex in $B$ on the other side will decrease the total number of straddling edges. Therefore, the number of straddling edges is minimized at $\frac{ab}{2}$.
\end{proof}

\begin{theorem}
If an underlying geograph has straddling span $w$, then it has at least $\frac{w}{2}$ vertices.
\end{theorem}

\begin{proof}
Choose the up direction along the line the produces the straddling span. We claim that no two straddling edges belong to the same chain. Indeed, if two edges are straddling, then their projections onto the $x$-axis must overlap at the point that is the projection of the line that produces the straddling span. But it is clear that the projections along the $x$-axis of the edges of any given chain must be mutually non-overlapping. Therefore, our geograph contains at least one chain for every straddling edge. Since there are at least $w$ straddling edges, the number of chains must be at least the same and the number of vertices must be at least $\frac{w}{2}$.
\end{proof}

\begin{corollary}
If an underlying geograph contains a $k$-clique, then it has at least $\lfloor k^2/2 \rfloor$ vertices. Consequently, the largest clique in the underlying graph with $n$ vertices cannot exceed $\sqrt{2n}+1$ vertices.
\end{corollary}

\begin{corollary}
If an underlying geograph contains a $(a,b)$-complete bipartite subgraph, then it has at least $ab$ vertices.
\end{corollary}

Note that now both the lower bound and the upper bound for the largest clique are on the order of $\sqrt{n}$.

\section{New Upper Bound for the Number of Halv\-ing Lines}\label{sec:upperbound}

First, we need the following notation: $C$ is the crossing number, which is the smallest number of crossings in a drawing of the graph on the plane, $E$ is the number of edges, and $n$ is the number of vertices.

The famous crossing number lemma \cite{PachToth} states:

\begin{lemma}
When $E$ is greater than $7.5n$, then $C \ge \frac{4E^3}{135n^2}$.
\end{lemma}

For large enough $n$ the best known construction yields far more than $7.5n$ halving lines, so we are not concerned with the case when $E \leq 7.5n$.

The current best upper bound for the number of halving lines relates the number of chains and the crossing number.

We begin by reintroducing some of Dey's definitions \cite{Dey98}. A \textit{common upper tangent} of a pair of chains is a line that passes through two vertices, one from each of the chains, and has the rest of the chains strictly below it. The upper tangents can easily be constructed by drawing the convex hull of two chains. Two chains can have several common upper tangents. 

The vertical line passing through a crossing intersects a unique common upper tangent. We say that this common upper tangent is \textit{charged} to the crossing. Two crossings from the same two chains cannot charge the same upper tangent. Dey \cite{Dey98} showed that even two crossings from different pairs of chains cannot charge the same tangent.

\begin{lemma}[Dey]\label{thm:Dey}
Each common tangent is charged only once for crossings over all pairs of chains.
\end{lemma}

For every crossing $C$, define $f(C,\theta)$ to be the pair of vertices that forms the common upper tangent to which the crossing $C$ is charged when the graph is oriented so that the polar angle $\theta$ is up. Then $f$ maps crossings and angles between $0$ and $2\pi$ to pairs of vertices. 

Dey's upper bound proof \cite{Dey98} relies on the fact that, for a fixed $\theta$, the function $f$ is injective over the set of all crossings. From this it follows that the number of crossings does not exceed the number of pairs of vertices: $\binom{n}{2}$. 
This in turn gives a bound for the number of edges.

We will improve this result, but we first need to define some terms.

Given a polygon $P$ containing an edge $AB$, and a broken line connecting $A$ and $B$, we define this broken line to be \textit{concave in} with respect to $P$ iff it appears concave down when the entire configuration is rotated so that $AB$ lies both horizontal and above the region inside $P$. Intuitively this means that the broken line curves into the polygon in a concave fashion.

\begin{lemma}\label{thm:polygon}
Consider a crossing of two chains $S$ and $T$. If we choose the triangle formed by the crossing and the endpoints of its upper tangent as our reference polygon, then the sections of $S$ and $T$ between the crossing and the endpoints of the upper tangent are concave in with respect to this triangle.
\end{lemma}

\begin{proof}
Chains are concave down by nature, so the lemma follows the definition of concave in.
\end{proof}

Now we proceed onto the main theorem:

\begin{theorem}
For any two distinct crossings $C_1$, $C_2$, and orientation angles $\theta$, $\phi$, the corresponding pairs of vertices do not coincide: $f(C_1,\theta) \ne f(C_2,\phi)$.
\end{theorem}

\begin{proof}
Let us fix the following convention: if a chain passes through a given point $X$, then \textit{the edges of the chain to the left of} $X$ refer to the edges that lies to the left of $X$ when the chain is oriented to be concave down. The same holds for \textit{the edges of the chain to the right of} $X$. Any edge that has $X$ as an endpoint is said to be \textit{immediately to the left (or right)} of $X$. This convention allows us to refer to the same edge or set of edges more conveniently when considering the same chain under a different orientation.

We start by assuming that two distinct crossings do map to the same edge. Orient the configuration so that the edge lies horizontal. Assume that crossings $C_1$, $C_2$ map to the same edge $AB$, where $A$ lies to the left of $B$. Then there are two chains $S_1$, $T_1$ which cross at $C_1$ and have $AB$ as their upper common tangent. There are also two chains $S_2$, $T_2$ which cross at $C_2$, and also have $AB$ as their upper tangent, possibly under another orientation. Without loss of generality, assume that $S_1$ and $S_2$ pass through $A$, and that $T_1$ and $T_2$ pass through $B$.

We first claim that either $S_1$ and $S_2$ have different edges immediately to the right of $A$, or $T_1$ and $T_2$ contain different edges immediately to the left of $B$. This claim is obvious if $S_1,T_1$ lie on the opposite side of $AB$ from $S_2,T_2$. However, the claim is nontrivial if they lie on the same side of $AB$, so assume that this is the case. For the sake of contradiction, suppose that $S_1$ and $S_2$ have the same edge to the right of $A$, and $T_1$ and $T_2$ have the same edge to the left of $B$. Chains $S_1$ and $T_1$ cross at $C_1$, and Lemma~\ref{thm:Dey} forbids chains $S_2$ and $T_2$ from having two crossings both charged to $AB$, so chains $S_2$ and $T_2$ cannot cross at $C_1$. Therefore, if we consider the edges of $S_2$ to the right of $A$, and the edges of $T_2$ to the left of $B$, it follows from Lemma 6.8 that one of $S_2$ or $T_2$ must end before reaching $C_1$. Hence, $C_2$ must lie either on the edges of $S_2$ to the left of $A$, or on the edges of $T_2$ to the right of $B$. Either way, there is a section of either $S_2$ or $T_2$ that is not concave in with respect to the triangle $\triangle C_2AB$. This violates Lemma~\ref{thm:polygon}, so our claim is in fact correct.

Without loss of generality $S_1$, $S_2$ do not share the same edge immediately to the right of $A$. Call these two distinct edges $AP$ and $AQ$ respectively.

Note that $A$ must have degree at least 3, so it has at least one more edge $AR$ which is neither $AP$ nor $AQ$. From Corollary~\ref{thm:half-plane}, we can also stipulate that the vectors $\vec{AP},\vec{AQ},\vec{AR}$ do not all lie on one half-plane. Now there are two cases to consider:

\textit{Case 1: $S_1$, $S_2$ lie on opposite sides of line $AB$}. Since $P$ and $Q$ are on different sides of $AB$, $R$ must lie on the same side as one of them. Without loss of generality $R$ lies on the same side as $Q$. See Figure~\ref{fig:case1}, where thin lines are continuation of edges corresponding halving lines.

\begin{figure}[htbp]
\begin{center}
\includegraphics[scale=0.4]{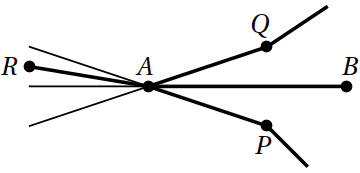}
\end{center}
  \caption{Case 1.}\label{fig:case1}
\end{figure}

Change the orientation to that of chain $S_1$. Because $\vec{AP},\vec{AQ},\vec{AR}$ do not all lie on one half-plane, one of $Q,A,P$ or $R,A,P$ must be concave down. Thus, if the chain $S_1$ did not cross $AB$, its wing at $A$ would contain either $AQ$ or $AR$, which is impossible because wings cannot contain edges inside them. But if $S_1$ did cross $AB$, then $AB$ would not be an upper tangent of $S_1$, which is equally contradictory.

\textit{Case 2: the two chains lie on the same side of line $AB$}. Without loss of generality $\angle PAB > \angle QAB$, see Figure~\ref{fig:case2}. Then it follows from the fact that $\vec{AP},\vec{AQ},\vec{AR}$ do not all lie on one half-plane that $R,A,P$ are concave down. As in the previous case, if $S_1$ did not cross $AB$, then its wing at $A$ would contain $AR$, and if it did cross $AB$, then $AB$ would not be an upper tangent. Both of these are impossible.

\begin{figure}[htbp]
\begin{center}
\includegraphics[scale=0.4]{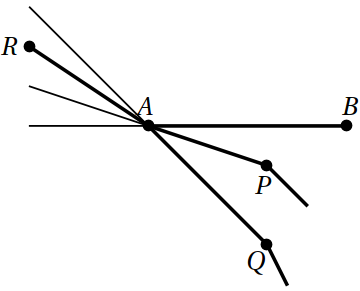}
\end{center}
  \caption{Case 2.}\label{fig:case2}
\end{figure}

We have exhausted all possibilities by contradiction, so our assumption that crossings $C_1$ and $C_2$ mapped to the same edge was false.
\end{proof}

Furthermore, we prove that each crossing maps to at least four distinct pairs of vertices.

\begin{lemma}
For a fixed crossing $C$, the function $f(C,\alpha)$ evaluates to at least four different vertex pairs as $\alpha$ ranges over four distinct angles determined by $C$.
\end{lemma}

\begin{proof}
Given a crossing, consider the two edges that determine the crossing. When extended, these two lines determine four distinct regions in the plane, with each region containing exactly one ray that bisects the angle between the lines. These four angle bisectors determine four possible directions to define as up. Fix one of these four orientations, and consider the upper common tangent of the crossing $C$. It is a line which passes through its corresponding region on the plane, and also its two neighboring regions, but not the region opposite to it. The same property holds for the three other common tangents. We thus conclude that the four possible common tangents determined by aligning $\alpha$ with the four angle bisectors must be distinct.
\end{proof}

\begin{theorem}
The number of halving lines is at most $\sqrt[3]{\binom{n}{2}\frac135{n^2}{16}}$.
\end{theorem}

\begin{proof}
The number of crossing does not exceed the number of pairs of vertices divided by 4:
$$C \le \binom{n}{2}/4.$$
Combining with the crossing lemma: 
$$\frac{4E^3}{135n^2} \le C,$$
we get 
$$E \le \sqrt[3]{\frac{135n^2}{16}\binom{n}{2}}.$$
This is approximately $1.62n^{4/3}$, which is an improvement over the current bound of $2.57n^{4/3}$.
\end{proof}

We constructed graphs with $n$ vertices and complete subgraphs of size $O(\sqrt{n})$. It follows from the crossing lemma that a complete graph of size $k$ has at least $O(k^4)$ crossings. Therefore, the maximum number of crossings in an underlying graph is of order $O(n^2)$. An $O(n^{\frac{4}{3}})$ bound is thus the optimal result we can get using the crossing lemma without any further restrictions on the graphs.

\section{Acknowledgements}

We are grateful to the PRIMES program at MIT for allowing us the opportunity to do this project, and to Prof. Jacob Fox (MIT) for suggesting the project.

\end{document}